\documentclass[11pt]{amsart}
\usepackage{amsmath,amssymb,amsthm}
\usepackage{graphicx}
\usepackage{tabu}

\usepackage[colorlinks=true,linkcolor=blue,citecolor=blue,pdfpagelabels=false]{hyperref}

\theoremstyle{plain}
  \newtheorem{theorem}{Theorem}[section]
  \newtheorem{lemma}[theorem]{Lemma}
  \newtheorem{conjecture}[theorem]{Conjecture}

\theoremstyle{definition}
  \newtheorem{example}[theorem]{Example}

\newenvironment{acknowledgements}{\bigskip\textbf{Acknowledgements.}}{}

\newcommand{\sumprime}{\mathop{\sum\nolimits'}}
\newcommand{\nequiv}{\not\equiv}
\renewcommand{\geq}{\geqslant}
\renewcommand{\leq}{\leqslant}
\newcommand{\scaleeq}[1]{\scalebox{0.8}{$\displaystyle #1$}}

\title{Supercongruences for sporadic sequences}

\author{Robert Osburn, Brundaban Sahu and Armin Straub}
\address{School of Mathematical Sciences, University College Dublin, Belfield, Dublin 4, Ireland}
\email{robert.osburn@ucd.ie}

\address{School of Mathematical Sciences, National Institute of Science Education and Research, Bhubaneswar 751005, India}
\email{brundaban.sahu@niser.ac.in}

\address{Department of Mathematics,
University of Illinois at Urbana-Champaign,
Urbana, IL 61801,
United States}
\curraddr{Max-Planck-Institut f\"ur Mathematik,
53111 Bonn,
Germany}
\email{astraub@illinois.edu}


\begin{document}

\dedicatory{Dedicated to Frits Beukers on the occasion of his 60th birthday}

\begin{abstract} 
We prove two-term supercongruences for generalizations of recently discovered sporadic sequences of Cooper. We also discuss recent progress and future directions concerning other types of supercongruences.
\end{abstract}

\date{\today}

\maketitle

\section{Introduction}

The term \textit{supercongruence} first appeared in Beukers' work
\cite{beukers-apery85} and was the subject of the Ph.D. thesis of Coster
\cite{coster}. It refers to the fact that congruences of a certain type are
stronger than those suggested by formal group theory. A motivating example in
\cite{beukers-apery85} and \cite{coster} is the Ap\'ery numbers
\begin{equation}
  A (n) = \sum_{k = 0}^n \binom{n}{k}^2 \binom{n + k}{k}^2 \label{eq:apery}
\end{equation}
which not only satisfy \cite{gessel}
\begin{equation}
  A (m p ) \equiv A (m) \pmod{p^3}, \label{ges}
\end{equation}
but the \textit{two-term supercongruence} \cite{coster}
\begin{equation}
  A (m p^r ) \equiv A (m p^{r - 1}) \pmod{p^{3 r}} \label{eq:apery-sc}
\end{equation}
for primes $p \geq 5$ and integers $m$, $r \geq 1$. In 1985, Beukers related
these numbers to the $p$-th Fourier coefficient $a (p)$ of $\eta^4 (2 z)
\eta^4 (4 z)$, where
\begin{equation*}
  \eta (z) = q^{1/24}  \prod_{n = 1}^{\infty} (1 - q^{n})
\end{equation*}
is the Dedekind eta function and $q=e^{2\pi iz}$, with $z$ in the upper half-plane. He proved that \cite{beukers}
\begin{equation}
  A \left( \frac{p - 1}{2} \right) \equiv a (p) \pmod{p}
  \label{modsuper}
\end{equation}
and then conjectured that \eqref{modsuper} holds modulo $p^2$. In \cite{ao},
Ahlgren and Ono proved this \textit{modular supercongruence} using Gaussian
hypergeometric series \cite{greene}. The techniques in \cite{ao} have been
the basis for several recent results of this type (see \cite{kilbourn},
\cite{lr}, \cite{dermot}, \cite{mort1}--\cite{mort4}, \cite{os}).
Other types of supercongruences are also of considerable interest. \textit{Ramanujan-type supercongruences} are $p$-adic versions of
formulas of Ramanujan which relate binomial sums to special values of the
gamma function (or $1 / \pi^a$, $a \geq 1$). For example, van Hamme
\cite{vh} conjectured that for Ramanujan's formula
\begin{equation*}
  \sum_{k = 0}^{\infty} (4 k + 1) \binom{- 1 / 2}{k}^5 = \frac{2}{\Gamma (3 /
  4)^4}, \label{ram}
\end{equation*}
we have the $p$-adic analogue
\begin{equation}
  \sum_{k = 0}^{(p - 1) / 2} (4 k + 1) \binom{- 1 / 2}{k}^5 \equiv \left\{
  \begin{array}{lll}
    - \frac{p}{\Gamma_p (3 / 4)^4} & \pmod{p^3}, & \text{if
    $p \equiv 1 \pmod{4}$,}\\
    0 & \pmod{p^3}, & \text{if $p \equiv 3 \pmod{4}$,}
  \end{array} \right. \label{padicram}
\end{equation}
where $\Gamma_p(\cdot)$ is the $p$-adic gamma function. For a proof of
\eqref{padicram}, see \cite{mo}. For recent progress in this direction, see
\cite{cdlns}, \cite{kilbourn}, \cite{g}, \cite{gz},  \cite{long}, \cite{mort} or
\cite{zudilin}. Finally, {\it Atkin--Swinnerton-Dyer}
supercongruences have been recently studied in \cite{cosvan}, \cite{klmsy}, \cite{lilong}
and \cite{scholl}.

In this paper, we consider the sequences of numbers given by
\begin{equation}
  s_7 (n) = \sum_{k = 0}^n \binom{n}{k}^2 \binom{n + k}{k} \binom{2 k}{n} 
  \label{eq:s7}
\end{equation}
as well as
\begin{equation} 
s_{18} (n) = \sum^{[n / 3]}_{k = 0} (- 1)^k \binom{n}{k} \binom{2 k}{k}
  \binom{2 (n - k)}{n - k} \left[ \binom{2 n - 3 k - 1}{n} + \binom{2 n - 3
  k}{n} \right],  \label{s18}
\end{equation}
with $s_{18} (0) = 1$.
These ``sporadic'' sequences were recently discovered by Cooper \cite{cooper}
while performing a numerical search for sequences which appear as coefficients
of series for $1 / \pi$ and of series expansions in $t$ of modular forms where
$t$ is a modular function. Here, the subscripts $7$ and $18$ are used in \eqref{eq:s7} 
and \eqref{s18} as the associated modular function is of level 7 and 18, respectively (see Theorem 3.1 in
\cite{cooper}). In \cite{cooper}, Cooper searched for parameters $(a, b,
c, d)$ such that the recurrence relation
\begin{equation}
  (n + 1)^3 s (n + 1) = (2 n + 1) (a n^2 + a n + b) s (n) - n (c n^2 + d) s (n
  - 1), \label{eq:csearch}
\end{equation}
with initial conditions $s (- 1) = 0$, $s (0) = 1$, produces only integer
values $s (n)$ for all $n \geq 0$. The tuple $(17, 5, 1, 0)$ corresponds
to the Ap\'ery numbers \eqref{eq:apery}, while the tuples $(13, 4, - 27, 3)$ and $(14, 6, - 192, 12)$ 
correspond to the sequences $s_7 (n)$ and $s_{18}(n)$, respectively. See \cite{avsz} for the case $d = 0$.
This search was motivated by Beukers' \cite{beukers3} and Zagier's \cite{zagier} work on sequences $t
(n)$ defined by
\begin{equation}
  (n + 1)^2 t (n + 1) = (an^2 + an + b) t (n) - cn^2 t (n - 1),
  \label{eq:zsearch}
\end{equation}
with initial conditions $t (- 1) = 0$, $t (0) = 1$, such that $t (n) \in
\mathbb{Z}$ for all $n \geq 0$. Zagier's search yielded six sequences
that are not either terminating, polynomial, hypergeometric or Legendrian.
These six sequences were called \textit{sporadic}.

Interestingly, Cooper conjectured the following
congruences (see Conjecture 5.1 in \cite{cooper}) which are reminiscent of
\eqref{ges}.

\begin{conjecture}
  \label{conj:s718-sc}For any prime $p \geq 3$,
  \begin{equation} 
    s_7 (m p ) \equiv s_7 (m) \pmod{p^3}  \label{eq:s7-sc}.
  \end{equation}
  Likewise, for any prime $p$,
  \begin{equation} \label{eq:s18-sc}
  s_{18}(mp) \equiv s_{18}(m) \pmod{p^2}.
  \end{equation}
\end{conjecture}

The purpose of this paper is to exhibit that \eqref{eq:s7-sc} and \eqref{eq:s18-sc} are special
cases of general two-term supercongruences. For integers $A$, $B$, $C$, let
\begin{equation}
 \mathcal{S} (n ; A, B, C) = \sum_{k = 0}^n \binom{n}{k}^A \binom{n + k}{k}^B
  \binom{2 k}{n}^C \label{eq:s7x}.
\end{equation}
Note that this family of sequences includes the Ap\'ery numbers as well as the sequence $s_7(n)$.

Our main results are the following supercongruences, the first of which, in particular, generalizes
the supercongruence \eqref{eq:apery-sc} for the Ap\'ery numbers.
\begin{theorem}
  \label{thm:sc-s7} Let $A \geq 2$ and $B,C \geq 0$ be integers. For any integers $m$, $r \geq 1$ and primes $p \geq 5$, we have
   \begin{equation}
    \mathcal{S} (m p^r ; A, B, C) \equiv \mathcal{S} (m p^{r - 1} ; A, B, C)
    \pmod{p^{3 r}}.  \label{eq:sc-ABC}
  \end{equation}
\end{theorem}

\begin{theorem}
  \label{thm:sc-s18}
  For any integers $m$, $r \geq 1$ and any primes $p$, we have
  \begin{equation}
    s_{18} (m p^r) \equiv s_{18} (m p^{r - 1}) \pmod{p^{2
    r}} . \label{eq:sc-s18}
  \end{equation}
\end{theorem}

Note that by taking $(A, B, C) = (2, 1, 1)$ and $r=1$ in
Theorem~\ref{thm:sc-s7}, we prove \eqref{eq:s7-sc} of
Conjecture~\ref{conj:s718-sc} for primes $p \geq 5$.  Moreover,
Theorem~\ref{thm:sc-s7} shows a general supercongruence for the sporadic
sequence 
\begin{equation*}
  \sum_{k = 0}^n \binom{n}{k}^2 \binom{2 k}{n}^2,
\end{equation*}
which is case $(\epsilon)$ in \cite{avsz} (see also Table \ref{tbl:sporadic3}
in Section~\ref{sec:concl}).
On the other hand, we remark that Theorem~\ref{thm:sc-s18} is considerably simpler than
Theorem~\ref{thm:sc-s7} because it suffices to consider each summand of the sum
\eqref{s18}, defining $s_{18}(n)$, individually. In both cases, our proof of
the congruences of Conjecture~\ref{conj:s718-sc} relies on the presence of the
binomial sums \eqref{eq:s7} and \eqref{s18} which were discovered by Zudilin.

Finally, we should mention that Cooper \cite{cooper} conjectures congruences
similar to \eqref{eq:s7-sc} for $p=2$, as well as a
stronger version of \eqref{eq:s18-sc} for $p=2,3$. These conjectures as well as 
\eqref{eq:s7-sc} for $p=3$ remain open. Based on numerical evidence, we actually conjecture that
\begin{equation}
  s_7 (m 2^r) \equiv s_7 (m 2^{r-1}) \pmod{2^{3r+2}}
\end{equation}
for $m\geq4$ and
\begin{equation}
  s_7 (m 3^{r}) \equiv s_7 (m 3^{r-1}) \pmod{3^{3r}}
\end{equation}
for any positive integer $m$, as well as
\begin{equation}
  s_{18}(m 2^r) \equiv s_{18}(m 2^{r-1}) \pmod{2^{2r+3}}
\end{equation}
for $m \geq 2$ and
\begin{equation}\label{eq:s18:mod3}
  s_{18}(m 3^r) \equiv s_{18}(m 3^{r-1}) \pmod{3^{3r-1}}
\end{equation}
for $m \geq 3$ (for $r=1$ the congruence \eqref{eq:s18:mod3} empirically holds
modulo $3^3$).
Slightly weaker congruences appear to hold in the cases when $m$ is not large
enough.
We expect that these conjectures, which naturally generalize the ones from
\cite{cooper}, can be established using the techniques we use in the case
$p\geq5$ when coupled with a careful and likely very technical analysis of the
kind presented at the end of our proof of Theorem~\ref{thm:sc-s18}.

The remainder of the paper is organized as follows. Section~\ref{sec:proof} is
devoted to the proofs of Theorems~\ref{thm:sc-s7} and \ref{thm:sc-s18}.
We then indicate in Section~\ref{sec:x} that these proofs readily
generalize to other sequences of interest.
In Section~\ref{sec:concl}, we conclude with remarks concerning future
directions.  In particular, we discuss both proven and conjectural two-term
supercongruences for all known sporadic sequences.

\section{Proof of Theorems \ref{thm:sc-s7} and \ref{thm:sc-s18}}\label{sec:proof}

Throughout this section, following \cite{beukers-apery85}, we let $\sumprime$ denote the sum over indices
not divisible by $p$.

We first recall the following version of Jacobsthal's binomial congruence
\cite{bsftalj}. For a proof, when $a,b\geq0$, we refer to \cite{gessel-j},
\cite{granville-bin97}, while the extension to negative integers is discussed
in \cite{s13}. Similar congruences hold \cite{gessel-j,s13} in the cases $p=2$
and $p=3$.

\begin{lemma}
  \label{lem:jacobsthal}For primes $p \geq 5$, integers $a,b$ and integers $r, s
  \geq 1$,
  \begin{equation}
    \binom{p^r a}{p^s b} / \binom{p^{r - 1} a}{p^{s - 1} b} \equiv 1
    \pmod{p^{r + s + \min (r, s)}} . \label{eq:jacobsthal}
  \end{equation}
\end{lemma}

We will also make use, in the case $n = - 2$, of the following simple
congruences.

\begin{lemma}
  \label{lem:powersummod}Let $p$ be a prime and $n$ an integer such that $n
  \nequiv 0$ modulo $p - 1$. Then, for all integers $r \geq 0$,
  \begin{equation}
    \sumprime_{k = 1}^{p^r-1} k^n \equiv 0 \pmod{p^r} . \label{eq:powersummod}
  \end{equation}
  If, additionally, $n$ is even, then, for primes $p \geq 5$,
  \begin{equation}
     \sumprime_{k = 1}^{(p^r - 1) / 2} \frac{1}{k^n} \equiv 0
    \pmod{p^r} . \label{eq:powersummod2}
  \end{equation}
\end{lemma}

\begin{proof}
  Since $n \nequiv 0$ modulo $p - 1$, we find an integer $\lambda$, not
  divisible by $p$, such that $\lambda^n \nequiv 1$ modulo $p$. Then,
  \begin{equation*}
    \lambda^n  \sumprime_{k = 1}^{p^r - 1} k^n =  \sumprime_{k = 1}^{p^r - 1} (\lambda k)^n \equiv  \sumprime_{k = 1}^{p^r - 1} k^n
     \pmod{p^r},
  \end{equation*}
  since the second and third sum run over the same residues modulo $p^r$. As
  $\lambda^n$ is not divisible by $p$, the congruence \eqref{eq:powersummod}
  follows.
  
  Congruence \eqref{eq:powersummod2} follows since the sum in
  \eqref{eq:powersummod2}, modulo $p^r$, is exactly half of the sum in
  \eqref{eq:powersummod} if $n$ is even.
\end{proof}

\begin{lemma}
  \label{lem:Bind}For integers $n$, $k \geq 1$ and $A, B, C \geq 0$, define
  \begin{equation*}
    \mathcal{B} (n, k) = \mathcal{B}(n,k; A,B,C)= \binom{n}{k}^A \binom{n + k}{k}^B \binom{2 k}{n}^C .
    \label{eq:Bnk}
  \end{equation*}
  Then, for primes $p \geq 5$ and integers $A \geq 2$, $r, s
  \geq 1$ and $k \geq 0$ such that $p \nmid k$,
  \begin{equation}
    \mathcal{B} (n p^r, k p^s ) \equiv \mathcal{B} (n p^{r - 1}, k p^{s - 1})
    \pmod{p^{3 r}} . \label{eq:Bind}
  \end{equation}
\end{lemma}

\begin{proof}
  By Jacobsthal's congruence \eqref{eq:jacobsthal}, we have
  \begin{equation*}
    \binom{n p^r}{k p^s} / \binom{n p^{r - 1}}{k p^{s - 1}} \equiv 1
     \pmod{p^{r + s + \min (r, s)}}
  \end{equation*}
  as well as
  \begin{equation*}
    \binom{n p^r  + k p^s}{n p^r} / \binom{n p^{r - 1} + k p^{s - 1}}{n p^{r -
     1}} \equiv 1 \pmod{p^{r + 2 \min (r, s)}}
  \end{equation*}
  and
  \begin{equation*}
    \binom{2 k p^s}{n p^r} / \binom{2k p^{s - 1}}{n p^{r - 1}} \equiv 1
     \pmod{p^{r + s + \min (r, s)}} .
  \end{equation*}
  Thus, if $s \geq r$ then congruence \eqref{eq:Bind} follows immediately
  upon applying Jacobsthal's congruence to each binomial coefficient. On the
  other hand, suppose $s \leq r$. Then the same approach yields
  \begin{equation}
    \mathcal{B} (n p^r, k p^s) = \lambda \mathcal{B} (n p^{r - 1}, k p^{s - 1}) \label{eq:Bind0}
  \end{equation}
  with $\lambda \equiv 1$ modulo $p^{r + 2 s}$. Moreover, since $p \nmid k$,
  we have
  \begin{equation*}
    \binom{n p^r}{k p^s} \equiv \binom{n p^{r - 1}}{k p^{s - 1}} \equiv 0
     \pmod{p^{r - s}} .
  \end{equation*}
  As $A \geq 2$, it follows that $p^{2 (r - s)}$ divides
  $\mathcal{B} (n p^r, k p^s)$. Since $r + 2 s + 2 (r - s) = 3 r$, congruence \eqref{eq:Bind} follows from \eqref{eq:Bind0}.
\end{proof}

In the following, $[x]$ denotes the largest integer $m$ such that $m \leq
x$.

\begin{lemma}
  \label{lem:binomind}For primes $p$, integers $m$ and integers $k \geq
  0$, $r \geq 1$,
  \begin{equation}
    \binom{m p^r  - 1}{k} (- 1)^k \equiv \binom{m p^{r - 1} - 1}{[k / p]} (-
    1)^{[k / p]} \pmod{p^r} . \label{eq:binomind}
  \end{equation}
\end{lemma}

\begin{proof}
  Following \cite[Lemma 2]{beukers-apery85}, we split the defining product
  of the binomial coefficient, according to whether the index is divisible by
  $p$ or not, to obtain
  \begin{eqnarray*}
    \binom{m p^r  - 1}{k} & = & \prod_{j = 1}^k \frac{m p^r  - j}{j}\\
    & = & \prod_{\substack{j = 1 \\ p \nmid j}}^k \frac{m p^r - j}{j} \prod_{\lambda =
    1}^{[k / p]} \frac{m p^{r - 1}  - \lambda}{\lambda}\\
    & = & \binom{m p^{r - 1}  - 1}{[k / p]} \prod_{\substack{j = 1 \\ p \nmid j}}^k
    \frac{m p^r  - j}{j} .
  \end{eqnarray*}
  The claim follows upon reducing modulo $p^r$.
\end{proof}

\begin{lemma}
  \label{lem:Cind}For integers $n, k, j$ and $A, B, C \geq 0$, define
  \begin{equation}
    \mathcal{C} (n, k, j) =  \mathcal{C} (n, k, j ; A, B, C)= \binom{n -
    1}{k}^A \binom{n + k}{k}^B \binom{j}{n}^C . \label{eq:Cnkj}
  \end{equation}
  Then, for primes $p$ and integers $n, k, j \geq 0$, $r \geq 1$,
  \begin{equation*}
    \mathcal{C} (n p^r, k, j) \equiv (- 1)^{(k + [k / p]) A} \mathcal{C} (np^{r
    - 1}, [k / p], [j / p]) \pmod{p^r} . \label{eq:Cind}
  \end{equation*}
\end{lemma}

\begin{proof}
  Note that
  \begin{equation} \label{eq:binomial}
    \binom{n + k}{k} = (- 1)^k \binom{- n - 1}{k} .
  \end{equation}
  Using Lemma~\ref{lem:binomind}, we find that
  \begin{equation*}
    (- 1)^k \binom{- n p^r  - 1}{k} \equiv (- 1)^{[k / p]} \binom{- n p^{r - 1}
      - 1}{[k / p]} \pmod{p^r}
  \end{equation*}
  or, equivalently,
  \begin{equation}
    \binom{n p^r  + k}{k} \equiv \binom{n p^{r - 1}  + [k / p]}{[k / p]}
    \pmod{p^r} . \label{eq:CindB}
  \end{equation}
  In particular,
  \begin{equation*}
    \binom{n p^r  + (j - n p^r)}{j - n p^r} \equiv \binom{n p^{r - 1} + [j /
     p] - n p^{r - 1}}{[j / p] - n p^{r - 1}} \pmod{p^r},
  \end{equation*}
  which is equivalent to
  \begin{equation}
    \binom{j}{n p^r} \equiv \binom{[j / p]}{n p^{r - 1}} \pmod{p^r} . \label{eq:CindC}
  \end{equation}
  The proof thus follows upon combining \eqref{eq:binomind}, \eqref{eq:CindB},
  \eqref{eq:CindC}.
\end{proof}

\begin{proof}[Proof of Theorem \ref{thm:sc-s7}]
  Adapting the original approaches of \cite{beukers-apery85} and \cite{gessel},
  we split the binomial sum \eqref{eq:s7x} as
  \begin{equation*}
    \mathcal{S} (m p^r ; A, B, C) = \sum_{s \geq 0} G_s (m p^r),
  \end{equation*}
  where
  \begin{equation*}
    G_s (n) =  \sumprime_k \mathcal{B}(n, kp^s).
    \end{equation*}
  It follows from Lemma \ref{lem:Bind} that, for $s \geq 1$,
  \begin{equation*}
    G_s (m p^r) \equiv G_{s - 1} (m p^{r - 1}) \pmod{p^{3
     r}} .
  \end{equation*}
  It therefore remains to show that
  \begin{equation}
    G_0 (m p^r) =  \sumprime_k  \mathcal{B}(mp^r, k) \equiv 0 \pmod{p^{3 r}} .
    \label{eq:G0}
  \end{equation}
  Note that
  \begin{equation}
    \binom{m p^r}{k} = \frac{m p^r}{k} \binom{m p^r - 1}{k - 1}
    \label{eq:binom1}
  \end{equation}
  is divisible by $p^r$ if $p \nmid k$. Hence, if $A \geq 3$ then
  \eqref{eq:G0} is obviously true and \eqref{eq:sc-ABC} follows.
  
  In the remainder, we consider the case $A = 2$. With \eqref{eq:binom1}
  substituted into \eqref{eq:G0}, we find that we need to show that
  \begin{equation}
    \sumprime_k \frac{1}{k^2} \binom{m p^r - 1}{k - 1}^2 \binom{m p^r  +
    k}{k}^B \binom{2 k}{m p^r}^C \equiv 0 \pmod{p^r} .
    \label{eq:G0k}
  \end{equation}
  If $p \nmid k$ then $[(k - 1) / p] = [k / p]$ so that, by
  Lemma~\ref{lem:binomind},
  \begin{equation}
    \binom{m p^r  - 1}{k - 1}^2 \equiv \binom{m p^{r - 1} - 1}{[k / p]}^2
    \pmod{p^r} . \label{eq:binomind2}
  \end{equation}
  By \eqref{eq:binomind2} and Lemma~\ref{lem:Cind} with $A = 0$, the left-hand
  side of \eqref{eq:G0k} is congruent modulo $p^r$ to
  \begin{equation} \label{step}
    \sumprime_k \frac{1}{k^2} \binom{m p^{r - 1} - 1}{[k / p]}^2
     \binom{mp^{r - 1} + [k / p]}{[k / p]}^B \binom{[2 k / p]}{m p^{r - 1}}^C.
  \end{equation}
  Using the notation of \eqref{eq:Cnkj} and (\ref{step}), congruence \eqref{eq:G0k} is
  equivalent to
  \begin{equation}
     \sumprime_k \frac{1}{k^2} \mathcal{C} (m p^{r - 1}, [k / p], [2 k /
    p]) \equiv 0 \pmod{p^r} . \label{eq:sumC1}
  \end{equation}
  In order to establish \eqref{eq:sumC1}, we now show that
  \begin{equation}
     \sumprime_k \frac{1}{k^2} \mathcal{C} (m p^r, k, 2 k) \equiv  \sumprime_k \frac{1}{k^2} \mathcal{C} (m p^{r - s}, [k / p^s], [2 k / p^s]) \pmod{p^r}
    \label{eq:sumCs0}
  \end{equation}
  for $s = 0$, $1$, $\ldots, r$. The case $s = 0$ is trivial, while the case $s = 1$ follows from
  Lemma~\ref{lem:Cind}. If we now let $\{k : p^s \} := k - p^s [k / p^s]$, the remainder of $k$ divided by $p^s$, then observe that
  \begin{equation*}
    [2 k / p^s] = 2 [k / p^s] + \left\{ \begin{array}{ll}
       1, & \text{if $\{k : p^s \} > p^s / 2$,}\\
       0, & \text{otherwise} .
     \end{array} \right.
  \end{equation*}
  Hence,
  \begin{align} \label{sumCsx}
    \sumprime_k \frac{1}{k^2} \mathcal{C} (m p^{r - s}, [k / p^s],
    [2 k / p^s]) & = \sum_n  \sumprime_{[k / p^s] = n} \frac{1}{k^2} \mathcal{C}
    (mp^{r - s} , [k / p^s], [2 k / p^s]) \nonumber\\
    & =  \sum_n \mathcal{C} (m p^{r - s}, n, 2 n)
     \sumprime_{\substack{[k / p^s] = n \\
      \{k : p^s \} < p^s / 2}} \frac{1}{k^2} \nonumber\\
    & + \sum_n \mathcal{C} (mp^{r - s}, n, 2 n + 1)
    \sumprime_{\substack{[k / p^s] = n \\
      \{k : p^s \} > p^s / 2}} \frac{1}{k^2} .  
  \end{align}
  It follows from \eqref{eq:powersummod2} of Lemma~\ref{lem:powersummod} that
  each of the inner sums in the last expression of (\ref{sumCsx}) is divisible by $p^s$. Suppose
  that $s < r$. Thus, by (\ref{sumCsx}) and Lemma~\ref{lem:Cind}, we now have
  \begin{align*}
   &  \sumprime_k \frac{1}{k^2}  \mathcal{C} (m p^{r - s}, [k / p^s],
    [2 k / p^s]) \\
    & \equiv \sum_n \mathcal{C}(m p^{r - s - 1}, [n / p], [2 n / p])
     \sumprime_{\substack{[k / p^s] = n \\
      \{k : p^s \} < p^s / 2}} \frac{1}{k^2} \\
      & +  \sum_n \mathcal{C} (m p^{r - s - 1}, [n / p], [(2 n + 1) / p])
     \sumprime_{\substack{[k / p^s] = n \\ \{k : p^s \} > p^s / 2}} \frac{1}{k^2} \pmod{p^r} \\
    & \equiv \sum_n  \sumprime_{\substack{[k / p^s] = n \\ \{k : p^s \} < p^s / 2}} \frac{1}{k^2} \mathcal{C}(m p^{r - s - 1}, [[k / p^s] / p], [[2 k / p^s] / p])\\
    &  + \sum_n  \sumprime_{\substack{[k / p^s] = n \\ \{k : p^s \} > p^s / 2}} \frac{1}{k^2} \mathcal{C} (m p^{r - s - 1}, [[k / p^s] / p], [[2 k / p^s] / p]) \pmod{p^r}\\
    & \equiv  \sumprime_k \frac{1}{k^2} \mathcal{C} (m p^{r - s - 1}, [k /
    p^{s + 1}], [2 k / p^{s + 1}]) \pmod{p^r}.
  \end{align*}
  Hence \eqref{eq:sumCs0} follows by induction on $s$. Moreover, the case $s =
  r$ in (\ref{sumCsx}) shows that, for $s = r$ and hence all $s = 0, 1,
  \ldots, r$, the sums in \eqref{eq:sumCs0} are divisible by $p^r$. In
  particular, the case $s = 1$ proves \eqref{eq:sumC1} and thus \eqref{eq:sc-ABC}. 
\end{proof}

\begin{proof}[Proof of Theorem \ref{thm:sc-s18}]
  As in \cite{os2}, we use the identity
  \begin{equation*}
    \binom{a - b}{c - d} \binom{b}{d} = \frac{\binom{a}{c} \binom{c}{d}
     \binom{a - c}{b - d}}{\binom{a}{b}},
  \end{equation*}
  to obtain that
  \begin{equation*}
    \binom{2 k}{k} \binom{2 (n - k)}{n - k} = \frac{\binom{2 n}{n}
     \binom{n}{k}^2}{\binom{2 n}{2 k}} = \binom{n}{k}  \frac{(2 k) ! (2 (n -
     k)) !}{n!k! (n - k) !} = \binom{n}{k} S (n - k, k),
  \end{equation*}
  where $S (m, n)$ are the {\emph{super Catalan numbers}}
  \begin{equation*}
    S (m, n) = \frac{(2 m) ! (2 n) !}{m!n! (m + n) !} .
  \end{equation*}
  We refer to \cite{gessel-ballot} and the references therein for the
  history and properties of these numbers. Here, we only need that $S (n - k,
  k)$ is an integer; in fact, it is an even integer if $n \geq 1$.
  
  Denote the summand of \eqref{s18} by $\mathcal{D} (n, k)$, that is
  \begin{equation*}
    \mathcal{D} (n, k) = (- 1)^k \binom{n}{k}^2 S (n - k, k) \left[ \binom{2
     n - 3 k - 1}{n} + \binom{2 n - 3 k}{n} \right],
  \end{equation*}
  so that
  \begin{equation*}
    s_{18} (m p^r) = \sum_{s \geq 0} \sumprime_k \mathcal{D} (m
     p^r, k p^s) .
  \end{equation*}
  In analogy with Lemma~\ref{lem:Bind}, we claim that, for primes $p \geq
  5$ and integers $r, s \geq 1$,
  \begin{equation}
    \mathcal{D} (m p^r, k p^s) \equiv \mathcal{D} (m p^{r - 1}, k p^{s - 1})
    \pmod{p^{3 r}} . \label{eq:Dind-s18}
  \end{equation}
  A direct application of Lemma~\ref{lem:jacobsthal} shows that
  \begin{equation*}
    \binom{2 m p^r - 3 k p^s}{m p^r} / \binom{2 m p^{r - 1} - 3 k p^{s -
     1}}{m p^{r - 1}} \equiv 1 \pmod{p^{r + 2 \min (r, s)}}
  \end{equation*}
  as well as
  \begin{equation*}
    S (m p^r - k p^s, k p^s) / S (m p^{r - 1} - k p^{s - 1}, k p^{s - 1})
     \equiv 1 \pmod{p^{r + 2 \min (r, s)}} .
  \end{equation*}
  On the other hand, for all integers $a$ and integers $b \geq 0$,
  \begin{equation}
    \binom{a + b}{b} = (- 1)^b \binom{- a - 1}{b},
  \end{equation}
  so that
  \begin{equation*}
    \binom{2 n - 3 k - 1}{n} = (- 1)^n \binom{3 k - n}{n} .
  \end{equation*}
  Hence, Lemma~\ref{lem:jacobsthal} implies that
  \begin{equation*}
    \binom{2 m p^r - 3 k p^s - 1}{m p^r} / \binom{2 m p^{r - 1} - 3 k p^{s -
     1} - 1}{m p^{r - 1}} \equiv 1 \pmod{p^{r + 2 \min (r,s)}} .
  \end{equation*}
  Proceeding as in the proof of Lemma~\ref{lem:Bind}, we therefore obtain
  \eqref{eq:Dind-s18}. The congruence \eqref{eq:sc-s18} then follows if we can
  prove that, for integers $k$ such that $p \nmid k$,
  \begin{equation*}
    \mathcal{D} (m p^r, k) \equiv 0 \pmod{p^{2 r}} .
  \end{equation*}
  This is an immediate consequence of the fact that $\mathcal{D} (n, k)$ is
  divisible by $\binom{n}{k}^2$.
  
  Finally, let us briefly indicate how to obtain the corresponding congruences
  in the case that $p = 2$ or $p = 3$. For $p = 3$, congruence
  \eqref{eq:jacobsthal} of Lemma~\ref{lem:jacobsthal} only holds modulo $p^{r
  + s + \min (r, s) - 1}$. The same arguments as above then show that
  congruence \eqref{eq:Dind-s18} holds modulo $p^{3 r - 1}$, and hence modulo
  $p^{2 r}$, for $p = 3$. The case $p = 2$ requires some more attention. In
  that case, the counterpart of congruence \eqref{eq:jacobsthal} is
  \begin{equation*}
    \binom{2^r a}{2^s b} / \binom{2^{r - 1} a}{2^{s - 1} b} \equiv
     \varepsilon \pmod{2^{r + s + \min (r, s) - 2}}
  \end{equation*}
  with $\varepsilon = - 1$, if $2^{r - 1} a \equiv 0$, $2^{s - 1} b \equiv 1$
  modulo $2$, and $\varepsilon = 1$ otherwise. Hence, applying the same
  arguments as for $p > 2$ shows that
  \begin{equation*}
    \mathcal{D} (2^r m, 2^s k) = \lambda \mathcal{D} (2^{r - 1} m, 2^{s - 1}
     k)
  \end{equation*}
  with $\lambda \equiv \pm 1$ modulo $2^{r + 2 \min (r, s) - 2}$ and, hence,
  $\lambda \equiv 1$ modulo $2$. Moreover, both sides are divisible by $2^{2
  \max (r - s, 0) + 1}$ because $\binom{2^r m}{2^s k}$ is divisible by $2^{r -
  s}$, if $r \geq s$, and the super Catalan numbers $S (n - k, k)$ are
  even when $n \geq 1$. In the cases $r = 1$ or $s = 1$, this suffices to
  conclude that \eqref{eq:Dind-s18} holds modulo $p^{2 r}$ for $p = 2$. On the
  other hand, if $r \geq 2$ and $s \geq 2$, then going through the
  above computations reveals that $\lambda \equiv 1$ modulo $2^{r + 2 \min (r,
  s) - 2}$, and hence modulo $2^{r + \min (r, s)}$. Together with the
  divisibility of $\mathcal{D} (2^r m, 2^s k)$ by $2^{r - s}$, if $r \geq
  s$, we again find that \eqref{eq:Dind-s18} holds modulo $p^{2 r}$ for $p =
  2$.
\end{proof}

\section{Comments on direct generalizations}\label{sec:x}

The approaches of the proof of Theorems~\ref{thm:sc-s7} and \ref{thm:sc-s18},
which are based on \cite{gessel} and \cite{beukers-apery85}, generalize
easily to other sequences.

\begin{example}\label{eg:eta}
  For instance, consider the sequence
  \begin{equation}\label{eq:eta}
  Z(n)= \sum_{k = 0}^n (- 1)^k \binom{n}{k}^3 \left( \binom{4 n -5 k - 1}{3 n} + \binom{4 n - 5 k}{3 n} \right),
  \end{equation}
  which is case $(\eta)$ in \cite{avsz} (see also Table \ref{tbl:sporadic3}
  in Section~\ref{sec:concl}).
  We claim that the proof of Theorem~\ref{thm:sc-s18} naturally extends to show that, for primes $p\geq5$,
  \begin{equation}\label{eq:sc-eta}
  Z(mp^r) \equiv Z(mp^{r-1}) \pmod{p^{3r}}.
  \end{equation}

  Similar to the proof of Theorem~\ref{thm:sc-s18}, write
  \begin{equation*}
  Z(mp^r)=\sum_{s \geq 0} \sumprime_k \mathcal{A}(mp^r, kp^s),
  \end{equation*}
  where 
  \begin{equation*}
  \mathcal{A}(n,k)= (- 1)^k \binom{n}{k}^3 \left( \binom{4 n -5 k - 1}{3 n} + \binom{4 n - 5 k}{3 n} \right).
  \end{equation*}
  As in the proof of Theorem~\ref{thm:sc-s18}, we find that for primes $p \geq 5$,
  \begin{equation*}
  \mathcal{A}(mp^r, kp^s) \equiv \mathcal{A}(mp^{r-1}, kp^{s-1}) \pmod{p^{3r}}.
  \end{equation*} 
  On the other hand, the presence of $\binom{n}{k}^3$ shows that, for $p \nmid k$,
  \begin{equation*}
  \mathcal{A}(mp^r, k) \equiv 0 \pmod{p^{3r}}.
  \end{equation*}
  Combining these two congruences, we conclude that the supercongruence \eqref{eq:sc-eta} indeed holds.
\end{example}

\begin{example}
  The proof of Theorem~\ref{thm:sc-s18} directly generalizes to
  supercongruences for the following family of sequences which includes
  $s_{18}(n)$ as the case $(A,B,C,D,E)=(1,1,1,1,1)$.  For nonnegative integers
  $A,B,C,D,E$, define $\mathcal{T} (n; A, B, C, D, E)$ to be the sequence
  \begin{equation*}
    \sum_{k = 0}^{[n / 3]} (- 1)^k
     \binom{n}{k}^{A} \binom{2 k}{k}^{B} \binom{2 (n - k)}{n - k}^{C}
     \left[ \binom{2 n - 3 k - 1}{n}^{D} + \binom{2 n - 3 k}{n}^{E} \right].
  \end{equation*}
  If $A\geq1$, $B\geq1$ and $C\geq1$, then
  \begin{equation*} \label{eq:sc-ABCDE}
   \mathcal{T} (m p^r; A, B, C, D, E) \equiv \mathcal{T} (m p^{r - 1};
     A, B, C, D, E) \pmod{p^{2 r}}
  \end{equation*}
  for all primes $p$. More generally, we have, for instance, that if $A\geq2$, $B\geq1$ and $C\geq1$, then
  \begin{equation*}
   \mathcal{T} (m p^r; A, B, C, D, E) \equiv \mathcal{T} (m p^{r - 1};
     A, B, C, D, E) \pmod{p^{3 r}}
  \end{equation*}
  for all primes $p\geq5$.
\end{example}

\begin{example}
  In \cite{cd-binomial06} and \cite{cd-binomial09}, the sequence
  \begin{equation*}
    u (n) = \sum_{k = 0}^n (- 1)^k \binom{n}{k} \binom{2 n}{k}
  \end{equation*}
  is studied. In particular, it is proved that $u (m p) \equiv u (m)$ modulo
  $p^3$ for all primes $p \geq 5$. Following the approach of
  Theorem~\ref{thm:sc-s7}, we obtain that, more generally,
  \begin{equation*}
    u (m p^r) \equiv u (m p^{r - 1}) \pmod{p^{3 r}}
  \end{equation*}
  for all primes $p \geq 5$. Let $a, b$ be nonnegative integers. It is
  shown in \cite{cd-binomial06} that the more general sequences
  \begin{equation*}
    u^{\varepsilon}_{a, b} (n) = \sum_{k = 0}^n (- 1)^{\varepsilon k}
     \binom{n}{k}^a \binom{2 n}{k}^b
  \end{equation*}
  satisfy, for primes $p \geq 5$, the congruence
  \begin{equation*}
    u^{\varepsilon}_{a, b} (p) \equiv u^{\varepsilon}_{a, b} (1) \pmod{p^3}
  \end{equation*}
  unless $(\varepsilon, a, b) = (0, 0, 1)$ or $(0, 1, 0)$. Again, we can
  generalize this congruence by using the approach of Theorem~\ref{thm:sc-s7}
  to show that, for primes $p \geq 5$,
  \begin{equation*}
    u_{a, b}^{\varepsilon} (m p^r) \equiv u_{a, b} (m p^{r - 1}) \pmod{p^{3 r}}
  \end{equation*}
  provided that $a + b \geq 2$.
\end{example}

\section{Concluding remarks}\label{sec:concl}

There are several directions for future work. First, we have numerically
checked that each of the six sporadic examples of Zagier (labelled \textbf{A},
\textbf{B}, \textbf{C}, \textbf{D}, \textbf{E}, \textbf{F} in \cite{zagier})
and the six sporadic examples in \cite{avsz} (labelled $(\delta)$, $(\eta)$,
$(\alpha)$, $(\epsilon)$, $(\zeta)$, $(\gamma)$) satisfies a two-term
supercongruence. Precisely, cases \textbf{A} and \textbf{D} are modulo $p^{3
r}$ while cases \textbf{B}, \textbf{C}, \textbf{E}, \textbf{F} are modulo $p^{2
r}$. All six cases from \cite{avsz} are modulo $p^{3r}$.  Cases \textbf{A} and
\textbf{D} have been proven by Coster \cite{coster} and cases \textbf{C} and
\textbf{E} were settled by the first two authors in \cite{os1} and \cite{os2}.
Cases $(\alpha)$ and $(\gamma)$ have been proven in \cite{os2} and
\cite{coster}, respectively. As mentioned in the introduction, the proof of
Theorem \ref{thm:sc-s7} implies case $(\epsilon)$. On the other hand, case $(\eta)$
is proven in Example~\ref{eg:eta} of Section~\ref{sec:x}. For a discussion concerning connections
between the six sporadic examples in \cite{avsz} and \cite{zagier}, see Theorem 4.1
in \cite{avsz} or Theorem 3.5 and Tables 1 and 2 in \cite{chancooper}.

The information on supercongruences for Ap\'ery-like numbers is summarized in
Tables~\ref{tbl:sporadic2} and \ref{tbl:sporadic3}. The value for $k$ indicates
that the sequence $A (n)$ (at least conjecturally in cases $\textbf{B}$,
$\textbf{F}$, $(\delta)$ and $(\zeta)$) satisfies the supercongruence
\begin{equation*}
  A (m p^r) \equiv A (m p^{r - 1}) \pmod{p^{k r}}
\end{equation*}
for primes $p \geq 5$. In the cases where this congruence has been proven, a
reference is indicated in the final column. It would be of interest to prove
cases $\textbf{B}$, $\textbf{F}$, $(\delta)$ and $(\zeta)$ and, more generally,
provide a framework for all two-term supercongruences.

\begin{table}[h]
  {\tabulinesep=1mm
  \begin{tabu}{|l|c|c|l|l|c|}
    \hline
    $(a, b, c)$ & \cite{zagier} & \cite{avsz} & $A (n)$ & $k$ & \\
    \hline
    $(7, 2, - 8)$ & {\textbf{A}} & (a) & \scaleeq{\sum_k \binom{n}{k}^3} & $3$ &
    \cite{coster}\\
    \hline
    $(11, 3, - 1)$ & {\textbf{D}} & (b) & \scaleeq{\sum_k \binom{n}{k}^2 \binom{n +
    k}{n}} & $3$ & \cite{coster}\\
    \hline
    $(10, 3, 9)$ & {\textbf{C}} & (c) & \scaleeq{\sum_k \binom{n}{k}^2 \binom{2
    k}{k}} & $2$ & \cite{os1}\\
    \hline
    $(12, 4, 32)$ & {\textbf{E}} & (d) & \scaleeq{\sum_k \binom{n}{k} \binom{2 k}{k}
    \binom{2 (n - k)}{n - k}} & $2$ & \cite{os2}\\
    \hline
    $(9, 3, 27)$ & {\textbf{B}} & (f) & \scaleeq{\sum_k (- 1)^k 3^{n - 3 k}
    \binom{n}{3 k} \frac{(3 k) !}{k!^3}} & $2$ & \\
    \hline
    $(17, 6, 72)$ & {\textbf{F}} & (g) & \scaleeq{\sum_{k, l} (- 1)^k 8^{n - k}
    \binom{n}{k} \binom{k}{l}^3} & $2$ & \\
    \hline
  \end{tabu}}
  \caption{\label{tbl:sporadic2}Sporadic sequences from \cite{zagier} for
  \eqref{eq:zsearch}}
\end{table}

\begin{table}[h]
  {\tabulinesep=1mm
  \begin{tabu}{|l|c|l|l|c|}
    \hline
    $(a, b, c, d)$ & \cite{avsz}, \cite{cooper} & $A (n)$ & $k$ & \\
    \hline
    $(7, 3, 81, 0)$ & ($\delta$) & \scaleeq{\sum_k (- 1)^k 3^{n - 3 k} \binom{n}{3 k}
    \binom{n + k}{n} \frac{(3 k) !}{k!^3}} & $3$ & \\
    \hline
    $(11, 5, 125, 0)$ & ($\eta$) &
    defined in \eqref{eq:eta}
& $3$ & \eqref{eq:sc-eta} \\
    \hline
    $(10, 4, 64, 0)$ & ($\alpha$) & \scaleeq{\sum_k \binom{n}{k}^2 \binom{2 k}{k}
    \binom{2 (n - k)}{n - k}} & $3$ & \cite{os2}\\
    \hline
    $(12, 4, 16, 0)$ & ($\epsilon$) & \scaleeq{\sum_k \binom{n}{k}^2 \binom{2 k}{n}^2}
    & $3$ & \eqref{eq:sc-ABC}\\
    \hline
    $(9, 3, - 27, 0)$ & ($\zeta$) & \scaleeq{\sum_{k, l} \binom{n}{k}^2 \binom{n}{l}
    \binom{k}{l} \binom{k + l}{n}} & $3$ & \\
    \hline
    $(17, 5, 1, 0)$ & ($\gamma$) & \scaleeq{\sum_k \binom{n}{k}^2 \binom{n + k}{n}^2}
    & $3$ & \cite{coster}\\
    \hline
    $(6, 2, - 64, 4)$ & $s_{10}$ & \scaleeq{\sum_k \binom{n}{k}^4} & $3$ &
    \cite{coster}\\
    \hline
    $(13, 4, - 27, 3)$ & $s_7$ & \scaleeq{\sum_k \binom{n}{k}^2 \binom{n + k}{k}
    \binom{2 k}{n}} & $3$ & \eqref{eq:sc-ABC}\\
    \hline
    $(14, 6, - 192, 12)$ & $s_{18}$ &
    defined in \eqref{s18}
& $2$ &  \eqref{eq:sc-s18} \\
    \hline
  \end{tabu}}
  \caption{\label{tbl:sporadic3}Sporadic sequences from \cite{avsz} and
  \cite{cooper} for \eqref{eq:csearch}}
\end{table}

Secondly, for each of the 15 sporadic cases in
Tables~\ref{tbl:sporadic2} and \ref{tbl:sporadic3}, one could ask if there
exists a modular form $f (z)$ whose $p$-th Fourier coefficient satisfies a
modular supercongruence. This is true for cases \textbf{D} \cite{ahlgren} and
$(\gamma)$ \cite{ao}. 

Finally, it also appears that all known Ramanujan-type series for $1 / \pi^a$,
$a \geq 1$, have a $p$-adic analogue which satisfies a Ramanujan-type
supercongruence. Different techniques have been employed as there is currently
no general explanation for this occurrence. For example, the first author and
McCarthy \cite{mo} utilized Gaussian hypergeometric series and Whipple's
transformation to prove (\ref{padicram}). Zudilin \cite{zudilin} proved several
Ramanujan-type supercongruences using the Wilf--Zeilberger method while Long
\cite{long} used a combination of combinatorial identities, $p$-adic analysis
and transformations and ``strange'' evaluations of ordinary hypergeometric
series due to Whipple, Gessel and Gosper. 

\begin{acknowledgements}
  The first author would like to thank Frits Beukers and Matthijs Coster for
  their continued interest. The second author is partially funded by SERB grant SR/FTP/MS-053/2012. The third author would like to thank the
  Max-Planck-Institute for Mathematics in Bonn, where he was in residence when
  this work was completed, for providing wonderful working conditions. Finally, we thank the referee for their helpful comments and suggestions.
\end{acknowledgements}

\end{document}